\documentclass[reqno]{amsart}
\usepackage{amsmath}
\usepackage{amssymb}
\usepackage{amsfonts}
\usepackage{amsthm}
\usepackage{mathtools}
\usepackage{leftindex}
\usepackage{stmaryrd}
\usepackage{amsthm}
\usepackage{cite}
\usepackage{tikz} 
\usetikzlibrary{tqft,calc}
\usetikzlibrary{cd}
\usetikzlibrary{positioning}
\usepackage{enumitem}
\usepackage{graphicx}
\usepackage[left=2.5cm,right=2.5cm,top=3cm,bottom=3cm]{geometry}
\usepackage{tikz-cd}
\usepackage{hyperref}
\usepackage{ulem}

\newtheorem{thmx}{Theorem}

\newcommand{\bunderbrace}[2]{%
	\begin{array}[t]{@{}c@{}}
		\underbrace{#1}\\
		#2
	\end{array}
}

\usetikzlibrary{decorations.markings}
\tikzset{double line with arrow/.style args={#1,#2}{decorate,decoration={markings,%
			mark=at position 0 with {\coordinate (ta-base-1) at (0,1pt);
				\coordinate (ta-base-2) at (0,-1pt);},
			mark=at position 1 with {\draw[#1] (ta-base-1) -- (0,1pt);
				\draw[#2] (ta-base-2) -- (0,-1pt);
}}}}

\allowdisplaybreaks[3]

\makeatletter
\@namedef{subjclassname@1991}{2020 Mathematics Subject Classification}
\makeatother

\newtheorem{theorem}{Theorem}
\newtheorem*{theorem*}{Theorem}
\newtheorem{proposition}[theorem]{Proposition}

\theoremstyle{definition}
\newtheorem{definition}[theorem]{Definition}
\newtheorem{example}[theorem]{Example}

\newtheorem*{question*}{Question}

\newcommand{\sll}[1]{\mkern-4mu\mathbin{/\mkern-5mu/}_{\mkern-4mu{#1}}}

\newcommand{\g}{\mathfrak{g}}
\newcommand{\h}{\mathfrak{h}}

\newcommand{\p}{\mathfrak{p}}

\renewcommand{\exp}{\mathrm{exp}}

\numberwithin{equation}{section}

\title[Universal families of twisted cotangent bundles]{Universal families of twisted cotangent bundles}

\author[Peter Crooks]{Peter Crooks}
\address[Peter Crooks]{Department of Mathematics and Statistics\\ Utah State University \\ 3900 Old Main Hill \\ Logan, UT 84322, USA}
\email{peter.crooks@usu.edu}

\begin{document}
	
	\begin{abstract} 
		We prove several results in incidence-theoretic Lie theory and Poisson geometry, with connections to twisted cotangent bundles, symmetric varieties, and Grothendieck--Springer resolutions. Our inquiry begins with a complex algebraic group $G$ and complex $G$-variety $X$. One can study the \textit{affine Hamiltonian Lagrangian (AHL) $G$-bundles} over $X$. Lisiecki indexes the isomorphism classes of such bundles in the case of a homogeneous $G$-variety $X=G/H$; the indexing set  is the set of $H$-fixed points $(\h^*)^H\subset\h^*$, where $\h$ is the Lie algebra of $H$. In very rough terms, one may regard $\psi\in(\mathfrak{h}^*)^H$ as labeling the isomorphism class of a \textit{$\psi$-twisted cotangent bundle} of $G/H$. These twisted cotangent bundles feature prominently in geometric representation theory and symplectic geometry.
		
		We introduce and examine the notion of a \textit{universal family} of AHL $G$-bundles over a $G$-variety $X$, as part of a broader program on Lie-theoretic and incidence-theoretic constructions of regular Poisson varieties. This family is defined to be a flat family $\pi:\mathcal{U}\longrightarrow Y$, in which $\mathcal{U}$ is a Poisson variety, the fibers of $\pi$ form a  complete list of representatives of the isomorphism classes of AHL $G$-bundles over $X$, and other pertinent properties are satisfied. Our first main result is the construction of a universal family of AHL $G$-bundles over a homogeneous base $X=G/H$, for connected $H$. In our second main result, we take $X$ to be a conjugacy class $\mathcal{C}$ of self-normalizing closed subgroups of $G$. We associate to $\mathcal{C}$ a regular Poisson variety $\mathcal{U}_{\mathcal{C}}$, defined in incidence-theoretic terms. Attention is paid to the case of conjugacy classes of normalizers of symmetric subgroups. In the case of a connected semisimple group $G$ and conjugacy class $\mathcal{C}$ of parabolic subgroups, our third main result relates $\mathcal{U}_{\mathcal{C}}$ to the partial Grothendieck--Springer resolution for $\mathcal{C}$. 
	\end{abstract}
	
	\maketitle
	\begin{scriptsize}
		\setcounter{secnumdepth}{2}
		\setcounter{tocdepth}{2}
	\end{scriptsize}
	
	\section{Introduction}
	We begin with three general themes around which this manuscript is centered.
	\subsection{Theme 1: Incidence varieties in Poisson geometry}\label{Subsection: Context and motivation}
	Several interesting and well-studied Poisson varieties are realizable in incidence-theoretic terms \cite{EvensLu1,EvensLu2,EvensLu3,Chriss,Kamnitzer,YakimovBrown,YakimovBrownII,Colarusso,EvensLi,CM,Leslie,Schrader,Safronov}. Prominent examples include the partial Grothendieck--Springer resolutions for a connected complex semisimple affine algebraic group with Lie algebra $\g$. One first associates the incidence variety \begin{equation}\mathfrak{g}_{\mathcal{C}}\coloneqq\{(\p,x)\in\mathcal{C}\times\g:x\in\p\}\end{equation} to each conjugacy class $\mathcal{C}$ of parabolic subalgebras of $\g$. Projection to $\g$ defines a morphism $\mu_{\mathcal{C}}:\mathfrak{g}_{\mathcal{C}}\longrightarrow\g$, called the partial Grothendieck--Springer resolution for $\mathcal{C}$. There exists a unique Poisson Hamiltonian $G$-variety structure on $\g_{\mathcal{C}}$ for which $\mu_{\mathcal{C}}$ is a moment map \cite[Proposition 5.3]{CM}. A more concrete description of this structure can be obtained by choosing a parabolic subgroup $P\subset G$ with Lie algebra in $\mathcal{C}$. The cotangent bundle $T^*(G/U(P))$ is then a Hamiltonian $G\times L(P)$-variety, where $U(P)$ is the unipotent radical of $P$ and $L(P)\coloneqq P/U(P)$. One has $\g_{\mathcal{C}}\cong T^*(G/U(P))/L(P)$ as Poisson Hamiltonian $G$-varieties. In short, the Hamiltonian $G$-variety structure on $T^*(G/U(P))$ naturally induces such a structure on $\g_{\mathcal{C}}$.
	
	\subsection{Theme 2: Twisted cotangent bundles of $G/U$ and regular Poisson varieties}\label{Subsection: Twisted cotangent 1}
	Let us specialize to the case in which $\mathcal{C}=\mathcal{B}$ is the flag variety of all Borel subalgebras of $\g$. Write $U(P)=U$ and $L(P)=T$ in the notation above. The group $T$ is a torus of dimension equal to the rank of $G$. This fact implies that the Poisson variety $\g_{\mathcal{B}}\cong T^*(G/U)/T$ is regular, meaning that any two of its symplectic leaves have the same dimension. The map $\mu_{\mathcal{B}}:\g_{\mathcal{B}}\longrightarrow\g$ is called the full Grothendieck--Springer resolution of $\g$.
	
	In addition to inducing the full Grothendieck--Springer resolution, $T^*(G/U)$ features in the literature on twisted cotangent bundles and geometric representation theory. One may form the twisted cotangent bundle $T^*(G/U)^{\psi}$ for $\psi$ a regular central character of $\mathfrak{u}$, the Lie algebra of $U$. We refer the reader to \cite{GinzburgKazhdan} for a more detailed discussion of the twisted cotangent bundles of $G/U$.
	
	\subsection{Theme 3: Twisted cotangent bundles of homogeneous varieties}\label{Subsection: Twisted cotangent 2}
	Suppose that a complex algebraic group $G$ acts algebraically on a smooth complex algebraic variety $X$. This action has a canonical lift to a Hamiltonian $G$-action on the cotangent bundle $T^*X$. At the same time, the bundle projection $T^*X\longrightarrow X$ is an affine Lagrangian fibration. The $G$-action on $T^*X$ is compatible with this Lagrangian fibration in the sense that the fibration is $G$-equivariant. Lisiecki \cite{Lisiecki} defines an \textit{affine Hamiltonian Lagrangian (AHL) $G$-bundle} over $X$ to be the following slight generalization of this setup: a symplectic Hamiltonian $G$-variety $Y$, together with a $G$-equivariant affine Lagrangian fibration $Y\longrightarrow X$. This notion features in several recent articles \cite{Torres1,Torres2}.
	
	It is natural to consider the set $\mathcal{A}\mathcal{H}\mathcal{L}_G(X)$ of isomorphism classes of AHL $G$-bundles over $X$. Lisiecki gives a concrete description of this set in the case of a homogeneous $G$-variety $X=G/H$, where $H\subset G$ is a closed subgroup with Lie algebra $\h$. One first considers the subspace $(\h^*)^H\subset\h^*$ of fixed points of the coadjoint representation of $H$. Each $\psi\in(\h^*)^H$ determines a \textit{$\psi$-twisted cotangent bundle} of $G/H$: this is defined to be the symplectic quotient variety $T^*G\sll{\psi}H$, taken with respect to the right translation action of $H$ on $T^*G$. The left translation action of $G$ on $T^*G$ then yields a residual Hamiltonian $G$-variety structure on $T^*G\sll{\psi}H$. Another fact is that the cotangent bundle map $T^*G\longrightarrow G$ induces a $G$-equivariant affine Lagrangian fibration $T^*G\sll{\psi}H\longrightarrow G/H$. In particular, $T^*G\sll{\psi}H$ is an AHL $G$-bundle over $G/H$. Lisiecki proves that the map
	\begin{equation}(\h^*)^H\longrightarrow\mathcal{A}\mathcal{H}\mathcal{L}_G(G/H),\quad\psi\mapsto[T^*G\sll{\psi}H]\end{equation} is a bijection.
	
	The discussion of $\mathcal{A}\mathcal{H}\mathcal{L}_G(X)$ motivates one to consider AHL $G$-bundles in \textit{universal families}. A universal family of AHL $G$-bundles over $X$ should arguably be defined as a flat, surjective morphism $\pi:\mathcal{U}\longrightarrow\mathcal{A}\mathcal{H}\mathcal{L}_G(X)$, where some appropriate algebraic variety structure on $\mathcal{A}\mathcal{H}\mathcal{L}_G(X)$ has been chosen. These data should come equipped with some Poisson-geometric features to make sense of $\pi^{-1}(\psi)$ being an AHL $G$-bundle for all $\psi\in\mathcal{A}\mathcal{H}\mathcal{L}_G(X)$. With this in mind, we require $\mathcal{U}$ to be a smooth Poisson variety carrying a Hamiltonian action of $G$ and moment map $\mu:\mathcal{U}\longrightarrow\g^*$. The morphism $\pi$ is also required to be $G$-invariant, and have the symplectic leaves of $\mathcal{U}$ as its fibers; this ensures that each fiber $\pi^{-1}(\psi)$ is a symplectic Hamiltonian $G$-variety with moment map $\mu\big\vert_{\pi^{-1}(\psi)}$. Other requirements are for $\mathcal{U}$ to come equipped with a $G$-equivariant morphism $\sigma:\mathcal{U}\longrightarrow X$, and for $\sigma\big\vert_{\pi^{-1}(\psi)}$ to give the Hamiltonian $G$-variety $\pi^{-1}(\psi)$ the structure of an AHL $G$-bundle for all $\psi\in\mathcal{A}\mathcal{H}\mathcal{L}_G(X)$. A final requirement is that this AHL $G$-bundle represent the isomorphism class $\psi\in\mathcal{A}\mathcal{H}\mathcal{L}_G(X)$. We then call the datum $(\mathcal{U},\mu,\pi,\sigma)$ a \textit{universal family} of AHL $G$-bundles over $X$.
	
	\subsection{Purpose} The purpose of this manuscript is to unify the themes mentioned above. We seek a uniform framework for studying the twisted cotangent bundles of homogeneous varieties, and regular Poisson varieties defined in Lie-theoretic and incidence-theoretic terms. Our hope is to stimulate new connections to Lie theory, especially those concerned with symmetric varieties. 
	
	\subsection{Main results}
	We construct a universal family of AHL $G$-bundles over $G/H$ for a connected closed subgroup $H\subset G$. Our approach is shown to yield a regular Poisson Hamiltonian $G$-variety $\mathcal{U}_{\mathcal{C}}$, where $\mathcal{C}$ is any conjugacy class of self-normalizing closed subgroups of $G$. This variety is defined in purely incidence-theoretic terms, and shown to make connections to symmetric subgroups of $G$. In the case of a connected semisimple $G$ and conjugacy class $\mathcal{C}$ of parabolic subgroups, we relate $\mathcal{U}_{\mathcal{C}}$ to the partial Grothendieck--Springer resolution for $\mathcal{C}$. The following is a more detailed overview.
	
	\subsubsection{A universal family of AHL $G$-bundles over $G/H$} Let $G$ be a complex algebraic group with Lie algebra $\g$. Suppose that $H\subset G$ is a closed subgroup with Lie algebra $\h\subset\g$. Note that the abelianization $A(H)\coloneqq H/[H,H]$ has Lie algebra $\mathfrak{a}(\h)\coloneqq\h/[\h,\h]$. At the same time, $T^*(G/[H,H])$ is a Hamiltonian $(G\times A(H))$-variety. The geometric quotient \begin{equation}\mathcal{U}_{G/H}\coloneqq T^*(G/[H,H])/A(H)\end{equation} exists, is smooth, and inherits a regular Poisson structure from the symplectic structure on $T^*(G/[H,H])$. It also inherits a residual Poisson Hamiltonian $G$-variety structure; the moment map $\mu_{G/H}:\mathcal{U}_{G/H}\longrightarrow\g^*$ is obtained by descending the $G$-moment map on $T^*(G/[H,H])$ to $\mathcal{U}_{G/H}$. The $A(H)$-moment map $T^*(G/[H,H])\longrightarrow\mathfrak{a}(\h)^*$ and cotangent bundle projection $T^*(G/[H,H])\longrightarrow G/[H,H]$ also descend to morphisms $\pi_{G/H}:\mathcal{U}_{G/H}\longrightarrow\mathfrak{a}(\h)^*$ and $\sigma_{G/H}:\mathcal{U}_{G/H}\longrightarrow G/H$, respectively. With these considerations in mind, the following is our first main result.
	
	\begin{thmx}\label{Theorem: Main1}
		If $H$ is connected, then $(\mathcal{U}_{G/H},\mu_{G/H},\pi_{G/H},\sigma_{G/H})$ is a universal family of AHL $G$-bundles over $G/H$.
	\end{thmx} 
	
	This result appears in the main text as Theorem \ref{Theorem: Universal}.
	
	\subsubsection{The regular Poisson varieties $\mathcal{U}_{\mathcal{C}}$} A conjugacy class $\mathcal{C}$ of closed subgroups of $G$ is called \textit{self-normalizing} if $N_G(H)=H$ for all $H\in\mathcal{C}$. In this case, we explain that $\mathcal{C}$ carries a distinguished algebraic variety structure (Proposition \ref{Proposition: Canonical structure on conjugacy classes}). We also form the $G$-invariant subset \begin{equation}\mathcal{U}_{\mathcal{C}}\coloneqq\{(H,\xi)\in\mathcal{C}\times\g^*:\xi\in[\h,\h]^{\circ}\}\subset\mathcal{C}\times\g^*,\end{equation} where $\h\subset\g$ is the Lie algebra of $H$ and $[\h,\h]^{\circ}\subset\g^*$ is the annihilator of $[\h,\h]\subset\g$. Given $H\in\mathcal{C}$, one has a canonical, $G$-equivariant bijection
	\begin{equation}\label{Equation: Canonical bijection}\mathcal{U}_{G/H}\overset{\cong}\longrightarrow\mathcal{U}_{\mathcal{C}}.\end{equation} The following is a coarse summary of Theorem \ref{Proposition: New}.
	
	\begin{thmx}\label{Theorem: Main2}
		Let $\mathcal{C}$ be a self-normalizing conjugacy class of closed subgroups of $G$. There exists a unique regular Poisson Hamiltonian $G$-variety structure on $\mathcal{U}_{\mathcal{C}}$ such that \eqref{Equation: Canonical bijection} is an isomorphism of Poisson Hamiltonian $G$-varieties for all $H\in\mathcal{C}$.
	\end{thmx}
	
	In the interest of applying Theorem \ref{Theorem: Main2}, we seek self-normalizing, closed subgroups of $G$. This task is simplified by assuming $G$ to be connected and semisimple. Parabolic subgroups are then self-normalizing; see Section \ref{Section: Twisted} for applications of Theorem \ref{Theorem: Main2} in this case. We also discuss the case of normalizers of symmetric subgroups; see Section \ref{Section: Symmetric}.
	
	\subsubsection{The case of parabolic subgroups}\label{Subsection: Parabolic case} Suppose that $G$ is connected and semisimple. Let $\mathcal{C}$ be a conjugacy class of parabolic subgroups of $G$. We may realize $\mathcal{U}_{\mathcal{C}}$ as \begin{equation}\mathcal{U}_{\mathcal{C}}=\{(P,x)\in\mathcal{C}\times\g:x\in[\p,\p]^{\perp}\},\end{equation} where $\p\subset\g$ is the Lie algebra of $P$ and $[\p,\p]^{\perp}\subset\g$ is the annihilator of $[\p,\p]$ under the Killing form. At the same time, one has the partial Grothendieck--Springer resolution \begin{equation}\mathfrak{g}_{\mathcal{C}}\coloneqq\{(P,x)\in\mathcal{C}\times\g:x\in\p\}.\end{equation} There is a set-theoretic inclusion $\mathcal{U}_{\mathcal{C}}\subset\g_{\mathcal{C}}$. The following appears in the main text as Theorem \ref{Proposition: Embedding}.
	
	\begin{thmx}\label{Theorem: Main3}
		The inclusion $\mathcal{U}_{\mathcal{C}}\subset\g_{\mathcal{C}}$ is one of Poisson Hamiltonian $G$-varieties.
	\end{thmx}
	
	\subsection{Organization} Each section begins with a summary of its contents. Sections \ref{Section: Lie-theoretic}, \ref{Section: Incidence varieties}, and \ref{Section: Twisted} are largely devoted to the proofs of Theorems \ref{Theorem: Main1}, \ref{Theorem: Main2}, and \ref{Theorem: Main3}, respectively. In Section \ref{Section: Symmetric}, we address the case of conjugacy classes of normalizers of symmetric subgroups.
	
	\subsection*{Acknowledgements} Gratitude is extended to an anonymous referee for carefully and constructively reviewing this manuscript. The author also thanks Michel Brion and Maxence Mayrand for useful conversations. He was supported by the National Science Foundation grant DMS-2454103 and Simons Foundation grant MP-TSM-00002292.
	
	\section{Lie-theoretic families of AHL $G$-bundles}\label{Section: Lie-theoretic}
	
	This section is concerned with the definition, existence, and properties of universal families of AHL $G$-bundles over a homogeneous base. In Subsection \ref{Subsection: Fundamental conventions}, we summarize the algebraic, geometric, and Lie-theoretic conventions used in our paper. Subsection \ref{Subsection: AHL bundles} subsequently provides a brief overview of AHL $G$-bundles \cite{Lisiecki}, and includes the formal definition of a universal family of AHL $G$-bundles over a $G$-variety. Lisiecki's description of $\mathcal{A}\mathcal{H}\mathcal{L}_G(G/H)$ is then reviewed in Subsection \ref{Subsection: A Lie-theoretic class}. The proof and basic setup of Theorem \ref{Theorem: Main1} constitute Subsection \ref{Subsection: Some universal families}.
	
	\subsection{Fundamental conventions}\label{Subsection: Fundamental conventions}
	Let us briefly outline the seminal conventions observed in our paper. The sole purpose of this outline is to preempt any possible ambiguity in subsequent parts of the paper. 
	
	\subsubsection*{Poisson varieties} We work exclusively over the field $\mathbb{C}$ of complex numbers. The term \textit{algebraic variety} refers to a reduced scheme $(X,\mathcal{O}_X)$ of finite type over $\mathbb{C}$; irreducibility is not assumed. This variety is called \textit{smooth} if $\dim T_{x_1}X=\dim T_{x_2}X$ for all $x_1,x_2\in X$. An algebraic bivector field $\sigma\in H^0(X,\wedge^2(TX))$ on a smooth algebraic variety $(X,\mathcal{O}_X)$ is called \textit{Poisson} if the bracket \begin{equation}\{\cdot,\cdot\}:\mathcal{O}_X\times\mathcal{O}_X\longrightarrow\mathcal{O}_X,\quad\{f_1,f_2\}\coloneqq \sigma(\mathrm{d}f_1\wedge\mathrm{d}f_2)\end{equation} renders $\mathcal{O}_X$ a sheaf of Poisson algebras. A \textit{Poisson variety} is understood to be a pair $(X,\sigma)$, consisting of a smooth algebraic variety $X$ and a Poisson bivector field $\sigma$ on $X$. 
	
	\subsubsection{Symplectic varieties} Given a Poisson variety $(X,\sigma)$, let us consider the vector bundle morphism
	\begin{equation}\sigma^{\vee}:T^*X\longrightarrow TX,\quad (x,\phi)\mapsto\sigma_x(\phi,\cdot)\in T_xX,\quad x\in X,\text{ }\phi\in T_x^*X.\end{equation} This Poisson variety is called \textit{symplectic} if $\sigma^{\vee}$ is a vector bundle isomorphism. One then has a unique closed, non-degenerate, algebraic two-form $\omega\in H^0(X,\wedge^2(T^*X))$ satisfying $(\sigma^{\vee})^{-1}=\omega^{\vee}$, where $\omega^{\vee}$ is defined by \begin{equation}\omega^{\vee}:TX\longrightarrow T^*X,\quad (x,v)\mapsto\omega_x(v,\cdot)\in T_x^*X,\quad x\in X,\text{ }v\in T_xX.\end{equation} A \textit{symplectic variety} is thereby a pair $(X,\omega)$ of a smooth algebraic variety $X$ and a closed, non-degenerate, algebraic two-form $\omega\in H^0(X,\wedge^2(T^*X))$, i.e. a \textit{symplectic form}. A closed subvariety $L\subset X$ is then called \textit{Lagrangian} if $L$ is smooth, $\dim L=\frac{1}{2}\dim X$, and $L$ satisfies $\iota^*\omega=0$ for the inclusion morphism $\iota:L\longrightarrow X$. 
	
	\subsubsection{Symplectic leaves} Each Poisson variety $(X,\sigma)$ determines a (possibly singular) holomorphic distribution $\{\mathrm{Image}(\sigma^{\vee}_x)\subset T_xX\}_{x\in X}$; its leaves are precisely the \textit{symplectic leaves} of $(X,\sigma)$. In other words, a symplectic leaf is a maximal injectively immersed, connected, complex submanifold $Y\subset X$ satisfying $T_xY=\mathrm{Image}(\sigma^{\vee}_x)$ for all $x\in Y$. The Poisson bivector field $\sigma$ is then tangent to $Y$; it thereby endows $Y$ with a holomorphic symplectic structure. 
	
	The \textit{rank} of $(X,\sigma)$ is the maximum of the symplectic leaf dimensions. This Poisson variety is called \textit{regular} if any two symplectic leaves have the same dimension.   
	
	\subsubsection{Algebraic groups and actions} Let $G$ be an affine algebraic group with Lie algebra $\g$. The nilpotent radical and derived subalgebra of $\g$ are denoted by $\mathfrak{u}(\mathfrak{g})\subset\mathfrak{g}$ and $[\mathfrak{g},\mathfrak{g}]\subset\g$, respectively. Let us also write $\mathfrak{a}(\mathfrak{g})\coloneqq\mathfrak{g}/[\mathfrak{g},\mathfrak{g}]$ for the abelianization of $\mathfrak{g}$. Analogous conventions are adopted on the group level: the unipotent radical, derived subgroup, and abelianization of $G$ are denoted by $U(G)\subset G$, $[G,G]\subset G$, and $A(G)\coloneqq G/[G,G]$, respectively.  
	
	We write
	\begin{equation}\mathrm{Ad}:G\longrightarrow\operatorname{GL}(\g)\quad\text{and}\quad\mathrm{Ad}^*:G\longrightarrow\operatorname{GL}(\g^*)\end{equation} for the adjoint and coadjoint representations of $G$, respectively. The adjoint and coadjoint representations of $\g$ are denoted by \begin{equation}\mathrm{ad}:\g\longrightarrow\mathfrak{gl}(\g)\quad\text{and}\quad\mathrm{ad}^*:\g\longrightarrow\mathfrak{gl}(\g^*),\end{equation} respectively. Let us also recall that the Killing form $\langle\cdot,\cdot\rangle:\g\otimes_{\mathbb{C}}\g\longrightarrow\mathbb{C}$ is $G$-invariant and defined by \begin{equation}\langle x,y\rangle=\mathrm{trace}(\mathrm{ad}_{x}\circ\mathrm{ad}_{y})\end{equation} for all $x,y\in\g$. Given a vector subspace $V\subset\g$, we write \begin{equation}V^{\perp}\coloneqq\{x\in\g:\langle x,y\rangle=0\text{ for all }y\in V\}\end{equation} for the annihilator of $V$ under the Killing form. The annihilator of $V$ in $\g^*$ is denoted by $V^{\circ}\subset\g^*$.
	
	All group actions are understood to be left group actions. The term \textit{$G$-variety} refers to a smooth algebraic variety $X$ carrying an algebraic $G$-action\footnote{While this smoothness requirement is non-standard, it greatly simplifies our exposition on forthcoming topics.}. Many examples arise through the \textit{associated bundle} construction. In more detail, let $H\subset G$ be a closed subgroup, and let $Y$ be an $H$-variety. Consider the algebraic action of $G\times H$ on $G\times Y$ given by \begin{equation}(g,h)\cdot(k,y)\coloneqq (gkh^{-1},h\cdot y),\quad (g,h)\in G\times H,\text{ }(k,y)\in G\times Y.\end{equation} The action of $H=\{e\}\times H\subset G\times H$ admits a smooth geometric quotient \begin{equation}G\times_H Y\coloneqq(G\times Y)/H;\end{equation} it carries a residual action of $G=G\times\{e\}\subset G\times H$. Projection from the first factor defines a $G$-equivariant morphism $G\times_H Y\longrightarrow G/H$, called the \textit{associated bundle} for $G$, $H$, and $Y$. 
	
	Suppose that $Y$ is a finite-dimensional $H$-module $V$. The associated bundle $G\times_H V\longrightarrow G/H$ is then a $G$-equivariant algebraic vector bundle of rank equal to $\dim V$. On the other hand, let $\mathcal{E}\longrightarrow G/H$ be any $G$-equivariant algebraic vector bundle with fiber over $[e]\in G/H$ equal to the $H$-module $V$. It follows that \begin{equation}G\times_H V\longrightarrow\mathcal{E},\quad [g:v]\mapsto g\cdot v\end{equation} is an isomorphism of $G$-equivariant algebraic vector bundles over $G/H$.
	
	\subsubsection{Algebraic Hamiltonian actions} Let $G$ be an affine algebraic group with Lie algebra $\g$ and exponential map $\exp:\g\longrightarrow G$. Suppose that $X$ is a $G$-variety. Each $\xi\in\g$ determines a \textit{generating vector field} $\xi_X$ on $X$ via
	\begin{equation}(\xi_X)_x=\frac{d}{dt}\bigg\vert_{t=0}\left(\exp(-t\xi)\cdot x\right)\in T_xX,\quad x\in X.\end{equation} A pair $(X,\mu)$ is called a \textit{Hamiltonian $G$-variety} if $X$ is a Poisson $G$-variety, the $G$-action on $X$ preserves the Poisson bivector field, and $\mu:X\longrightarrow\g^*$ is a $G$-equivariant morphism satisfying \begin{equation}\sigma^{\vee}\bigg(\mathrm{d}\big(\mu(\cdot)(\xi)\big)\bigg)=\xi_X\end{equation} for all $\xi\in\g$. In this case, $\mu$ is called the \textit{moment map}. An algebraic $G$-action on a Poisson variety $X$ is called \textit{Hamiltonian} if $(X,\mu)$ is a Hamiltonian $G$-variety for some moment map $\mu:X\longrightarrow\g^*$.
	
	Let us briefly discuss some Hamiltonian $G$-varieties that are ubiquitous in this paper. To begin, let $X$ be an arbitrary $G$-variety. The action of $G$ on $X$ has a canonical lift to an algebraic $G$-action on $T^*X$; the latter is called the \textit{cotangent lift} of the former. At the same time, $T^*X$ carries the tautological one-form; its exterior derivative is the canonical symplectic form on $T^*X$. The aforementioned $G$-action on $T^*X$ is then Hamiltonian with moment map \begin{equation}\mu:T^*X\longrightarrow\g^*,\quad \mu(x,\phi)(\xi)=-\phi((\xi_X)_x),\quad x\in X,\text{ }\phi\in T_x^*X,\text{ }\xi\in\g.\end{equation}
	
	\subsubsection{Algebraic symplectic quotients}
	We adopt the following specific version of Marsden--Weinstein reduction \cite{MarsdenWeinstein} in the algebraic category. Fix an algebraic group $G$ with Lie algebra $\g$. Let $G_{\xi}\subset G$ denote the $G$-stabilizer of $\xi\in\g^*$ under the coadjoint action. Given a symplectic Hamiltonian $G$-variety $(X,\mu)$, it follows that the closed subvariety $\mu^{-1}(\xi)\subset X$ is $G_{\xi}$-invariant. Now suppose that $G_{\xi}$ acts freely on $\mu^{-1}(\xi)$, and that this action admits a geometric quotient $\pi:\mu^{-1}(\xi)\longrightarrow\mu^{-1}(\xi)/G_{\xi}$. The quotient variety $\mu^{-1}(\xi)/G_{\xi}$ is then smooth and admits a unique algebraic symplectic form $\omega_{\xi}$ satisfying $\pi^*\omega_{\xi}=\mathrm{j}^*\omega$, where $\mathrm{j}:\mu^{-1}(\xi)\longrightarrow X$ is the inclusion and $\omega$ is the symplectic form on $X$ \cite[Theorem 5.3]{CrooksMayrand}. We set
	\begin{equation}X\sll{\xi}G\coloneqq\mu^{-1}(\xi)/G_{\xi},\end{equation} and call this symplectic variety \textit{the symplectic quotient of $X$ by $G$ at level $\xi$}.
	
	\subsection{Affine Hamiltonian Lagrangian $G$-bundles}\label{Subsection: AHL bundles}
	We now assemble some of the pertinent ideas and definitions from Sections 1 and 2 of Lisiecki's paper \cite{Lisiecki}. While Lisiecki works in the holomorphic category, our exposition reflects the slight modifications needed for these ideas and definitions to apply in the algebraic category.  
	
	\begin{definition}
		A morphism $\sigma:X\longrightarrow Y$ from a symplectic variety $(X,\omega)$ to a smooth variety $Y$ is called a \textit{Lagrangian fibration} if the following conditions are satisfied:
		\begin{itemize}
			\item[\textup{(i)}] $\sigma$ is a locally trivial holomorphic fiber bundle with respect to the Euclidean topologies on $X$ and $Y$;
			\item[\textup{(ii)}] for all $y\in Y$, the fiber $\sigma^{-1}(y)\subset X$ is a Lagrangian subvariety.
		\end{itemize}
	\end{definition}
	
	Suppose that this definition is satisfied. Given $x\in X$, consider the differential $\mathrm{d}\sigma_x:T_xX\longrightarrow T_{\sigma(x)}Y$ and its transpose $\mathrm{d}\sigma_x^*:T_{\sigma(x)}^*Y\longrightarrow T_{x}^*X$. Let us also recall that \begin{equation}\omega_x^{\vee}:T_xX\longrightarrow T_x^*X,\quad \omega_x^{\vee}(v)=\omega_x(v,\cdot)\end{equation} is a vector space isomorphism. It is straightforward to verify that $(\omega_x^{\vee})^{-1}(d\sigma_x^*(\phi))\in T_x\sigma^{-1}(\sigma(x))$ for all $\phi\in T_{\sigma(x)}^*Y$, and that \begin{equation}\label{Equation: Data}T_{\sigma(x)}^*Y\longrightarrow T_x\sigma^{-1}(\sigma(x)),\quad\phi\mapsto (\omega_x^{\vee})^{-1}(d\sigma_x^*(\phi))\end{equation} is a vector space isomorphism. By viewing $T^*Y\longrightarrow Y$ as a Lie algebroid with zero anchor map and abelian isotropy algebras, we may interpret \eqref{Equation: Data} as a Lie algebroid action of $T^*Y$ on $\sigma:X\longrightarrow Y$. It is then reasonable to distinguish the cases in which \eqref{Equation: Data} integrates to an action of the abelian group scheme $T^*Y\longrightarrow Y$. 
	
	\begin{definition}
		A Lagrangian fibration $\sigma:X\longrightarrow Y$ is called \textit{affine} if \eqref{Equation: Data} integrates to an action of the abelian group scheme $T^*Y\longrightarrow Y$ on $\sigma:X\longrightarrow Y$.
	\end{definition}
	
	It is advantageous to give the following enrichment of this definition in the context of Hamiltonian geometry.
	
	\begin{definition}
		Consider an algebraic group $G$ and $G$-variety $Y$.
		\begin{itemize}
			\item[\textup{(i)}] An \textit{affine Hamiltonian Lagrangian (AHL) $G$-bundle} over $Y$ is a triple $(X,\mu,\sigma)$, consisting of a symplectic Hamiltonian $G$-variety $(X,\mu)$ and a $G$-equivariant affine Lagrangian fibration $\sigma:X\longrightarrow Y$.
			\item[\textup{(ii)}] Two AHL $G$-bundles $(X_1,\mu_1,\sigma_1)$ and $(X_2,\mu_2,\sigma_2)$ over $Y$ are called \textit{isomorphic} if there exists a $G$-equivariant symplectic variety isomorphism $\varphi:X_1\longrightarrow X_2$ satisfying $\mu_1=\mu_2\circ\varphi$ and $\sigma_1=\sigma_2\circ\varphi$. 
			\item[\textup{(iii)}] We write $\mathcal{A}\mathcal{H}\mathcal{L}_G(Y)$ for the set of isomorphism classes of AHL $G$-bundles over $Y$.
		\end{itemize}
	\end{definition}
	
	Perhaps the most basic example of an AHL $G$-bundle over $Y$ is the Hamiltonian $G$-variety $T^*Y$; the integration of the Lie algebroid action is addition of covectors. Further examples are given in Subsection \ref{Subsection: A Lie-theoretic class}.
	
	The preceding definition combines with the last paragraph of Subsection \ref{Subsection: Context and motivation} to yield the following definition. 
	
	\begin{definition}\label{Definition: Universal}
		Consider an algebraic group $G$ and $G$-variety $Y$. A \textit{universal family} of AHL $G$-bundles over $Y$ consists of a quadruple $(\mathcal{U},\mu,\pi,\sigma)$ and an algebraic variety structure on $\mathcal{A}\mathcal{H}\mathcal{L}_G(Y)$ that satisfy the following properties:
		\begin{itemize}
			\item[\textup{(i)}] $(\mathcal{U},\mu)$ is a Poisson Hamiltonian $G$-variety;
			\item[\textup{(ii)}] $\pi:\mathcal{U}\longrightarrow\mathcal{A}\mathcal{H}\mathcal{L}_G(Y)$ is a surjective, flat, $G$-invariant morphism;
			\item[\textup{(iii)}] for all $\psi\in\mathcal{A}\mathcal{H}\mathcal{L}_G(Y)$, $\pi^{-1}(\psi)$ is a symplectic leaf of $\mathcal{U}$;
			\item[\textup{(iv)}] $\sigma:\mathcal{U}\longrightarrow Y$ is a $G$-equivariant morphism;
			\item[\textup{(v)}] for all $\psi\in\mathcal{A}\mathcal{H}\mathcal{L}_G(Y)$, \begin{equation}\left(\pi^{-1}(\psi),\mu\big\vert_{\pi^{-1}(\psi)},\sigma\big\vert_{\pi^{-1}(\psi)}\right)\end{equation} is an AHL $G$-bundle over $Y$ that represents the isomorphism class $\psi$.
		\end{itemize} 
	\end{definition}
	
	\subsection{Lisiecki's classification}\label{Subsection: A Lie-theoretic class}
	Suppose that we have an algebraic group $G$ with Lie algebra $\g$, a closed subgroup $H\subset G$ with Lie algebra $\h\subset\g$, and an element $\psi\in(\h^*)^H$. Such Lie-theoretic data determine an AHL $G$-bundle over $(T^*(G/H))^{\psi}$ over $G/H$, as detailed in \cite[Subsection 2.4]{Lisiecki}. To this end, use the left trivialization to identify the cotangent bundle $T^*G$ with $G\times\g^*$. The $(G\times G)$-action
	\begin{equation}(g_1,g_2)\cdot(h,\xi)\coloneqq(g_1hg_2^{-1},\mathrm{Ad}^*_{g_2}(\xi)),\quad (g_1,g_2)\in G\times G,\text{ }(h,\xi)\in G\times\g^*=T^*G\end{equation} is the cotangent lift of the following $(G\times G)$-action on $G$: \begin{equation}\label{Equation: Action on base}(g_1,g_2)\cdot h=g_1hg_2^{-1},\quad (g_1,g_2)\in G\times G,\text{ }h\in G.\end{equation} It follows that this lifted action is Hamiltonian with moment map
	\begin{equation}\label{Equation: Double moment}(\mu_1,\mu_2):G\times\g^*\longrightarrow\g^*\times\g^*=(\g\times\g)^*,\quad (g,\xi)\mapsto(\mathrm{Ad}_g^*(\xi),-\xi),\quad (g,\xi)\in G\times\g^*=T^*G.\end{equation} Observe that $\mu_1$ (resp. $\mu_2$) is a moment map for the Hamiltonian action of $G=G\times\{e\}$ (resp. $G=\{e\}\times G$) on $T^*G$. 
	
	Consider the action of the subgroup $H=\{e\}\times H\subset G\times G$ on $T^*G$. The symplectic quotient \begin{equation}(T^*(G/H))^{\psi}\coloneqq T^*G\sll{\psi}H=(\mathrm{r}\circ\mu_2)^{-1}(\psi)/H=G\times_H(\mathrm{r}^{-1}(-\psi))\end{equation} carries the following structure of an AHL $G$-bundle over $G/H$, where $\mathrm{r}:\g^*\longrightarrow\h^*$ is the restriction map: the Hamiltonian $G$-action is by left multiplication in the first factor of $G\times_H(\mathrm{r}^{-1}(-\psi))$, the moment map is
	\begin{equation}\label{Equation: Twisted moment} G\times_H(\mathrm{r}^{-1}(-\psi))\longrightarrow\g^*,\quad [g:\xi]\mapsto\mathrm{Ad}_g^*(\xi),\end{equation} and the affine Lagrangian fibration \begin{equation}\label{Equation: Twisted fibration}G\times_H(\mathrm{r}^{-1}(-\psi))\longrightarrow G/H,\quad [g:\xi]\mapsto[g]\end{equation} is projection from the first factor. 
	
	\begin{definition}\label{Definition: AHL1}
		Given $\psi\in(\mathfrak{h}^*)^H$, the AHL $G$-bundle $(T^*(G/H))^{\psi}$ is called the \textit{$\psi$-twisted cotangent bundle} of $G/H$.
	\end{definition}
	
	This nomenclature appears in \cite[Subsection 1.4]{Chriss} and reflects the observation that $(T^*(G/H))^{\psi}=T^*(G/H)$ for $\psi=0$. Lisiecki's construction of $(T^*(G/H))^{\psi}$ is inspired by the \textit{symplectic induction} procedure \cite{WeinsteinUniversal,GuilleminSternbergProceedings,Duval,KKS}. It also combines with the preceding discussion to give context for the following restatement of \cite[Proposition 2.6]{Lisiecki}.
	
	\begin{proposition}[Lisiecki]\label{Proposition: Lisiecki}
		The map \begin{equation}\label{Equation: Lisiecki}(\mathfrak{h}^*)^H\longrightarrow\mathcal{A}\mathcal{H}\mathcal{L}_G(G/H),\quad\psi\mapsto[(T^*(G/H))^{\psi}]\end{equation} is a bijection.
	\end{proposition}
	
	While Lisiecki proves Proposition \ref{Proposition: Lisiecki} in the holomorphic category, is it straightforward to deduce that this result holds in the algebraic setting of our paper.
	
	\subsection{Some universal families}\label{Subsection: Some universal families}
	Let $G$, $H$, $\g$, and $\h$ be as in the previous subsection. Observe that \begin{equation}\label{Equation: Statement}T^*(G/[H,H])=T^*G\sll{0}[H,H]=G\times_{[H,H]}[\h,\h]^{\circ},\end{equation} where the middle term denotes the symplectic quotient of $T^*G=G\times\g^*$ by $[H,H]=\{e\}\times[H,H]\subset G\times G$ at level zero. On the other hand, the $(G\times G)$-action on $G$ determines a residual action of $G\times A(H)$ on $G/[H,H]$:
	\begin{equation}\label{Equation: Second action on base}(g,[h])\cdot [k]=[gkh^{-1}],\quad (g,[h])\in G\times A(H),\text{ }[k]\in G/[H,H].\end{equation} The cotangent lift of the latter action to $T^*(G/[H,H])$ is the residual action of $G\times A(H)$ on the subquotient $T^*G\sll{0}[H,H]$ of $T^*G$, i.e.
	\begin{equation}\label{Equation: Lifted action}(g,[h])\cdot [k:\xi]\coloneqq [gkh^{-1}:\mathrm{Ad}_h^*(\xi)]\end{equation} for all $(g,[h])\in G\times A(H)$ and $[k:\xi]\in G\times_{[H,H]}[\h,\h]^{\circ}=T^*(G/[H,H])$. This and \eqref{Equation: Double moment} imply that \begin{equation}\label{Equation: Second double}(\mu_G,\mu_{A(H)}):T^*(G/[H,H])\longrightarrow\g^*\times\mathfrak{a}(\h)^*=(\g\times\mathfrak{a}(\h))^*,\quad [g:\xi]\mapsto(\mathrm{Ad}_g^*(\xi),-\xi\big\vert_{\h})\end{equation} is the canonical moment map for the cotangent lift of \eqref{Equation: Second action on base}, where one uses the identification $\mathfrak{a}(\h)^*=\mathrm{ann}_{\h^*}([\h,\h])$ to interpret $\xi\big\vert_{\h}\in\mathrm{ann}_{\h^*}([\h,\h])$ as belonging to $\mathrm{a}(\h)^*$. 
	
	Equations \eqref{Equation: Statement} and \eqref{Equation: Lifted action} tell us that the action of $A(H)=\{e\}\times A(H)\subset G\times A(H)$ on $T^*(G/[H,H])$ admits a smooth geometric quotient of
	\begin{equation}\label{Equation: Nicest}\mathcal{U}_{G/H}\coloneqq T^*(G/[H,H])/A(H)=G\times_H[\h,\h]^{\circ}.\end{equation} The cotangent bundle projection $T^*(G/[H,H])\longrightarrow G/[H,H]$ then descends to the associated vector bundle projection \begin{equation}\sigma_{G/H}:\mathcal{U}_{G/H}=G\times_H[\h,\h]^{\circ}\longrightarrow G/H,\quad [g:\xi]\mapsto [g].\end{equation} At the same time, $\mu_G$ and $\mu_{A(H)}$ descend to the $G$-equivariant morphism
	\begin{equation}\label{Equation: Moment map definition}\mu_{G/H}:\mathcal{U}_{G/H}\longrightarrow\g^*,\quad [g:\xi]\mapsto\mathrm{Ad}_g^*(\xi)\end{equation} and
	$A(H)$-invariant morphism \begin{equation}\pi_{G/H}:\mathcal{U}_{G/H}\longrightarrow\mathfrak{a}(\h)^*,\quad [g:\xi]\mapsto -\xi\big\vert_{\h},\end{equation} respectively. 
	
	\begin{proposition}\label{Proposition: Prelim}
		The following statements hold.
		\begin{itemize}
			\item[\textup{(i)}] The variety $\mathcal{U}_{G/H}$ carries a unique Poisson structure for which the quotient morphism $T^*(G/[H,H])\longrightarrow\mathcal{U}_{G/H}$ is a Poisson submersion.
			\item[\textup{(ii)}] This Poisson structure is regular and of rank $2\dim(G/H)$.
			\item[\textup{(iii)}] The action of $G=G\times\{e\}\subset G\times A(H)$ on $T^*(G/[H,H])$ descends to a Hamiltonian action of $G$ on $\mathcal{U}_{G/H}$.
			\item[\textup{(iv)}] The morphism $\mu_{G/H}$ is a moment map for the Hamiltonian action in \textup{(iii)}.
			\item[\textup{(v)}] For each $\psi\in\mathfrak{a}(\h)^*$, $\pi_{G/H}^{-1}(\psi)$ is a Poisson subvariety of $\mathcal{U}_{G/H}$ satisfying \begin{equation}\pi_{G/H}^{-1}(\psi)=T^*(G/[H,H])\sll{\psi}A(H)\end{equation} as closed subvarieties of $\mathcal{U}_{G/H}$. The Poisson structure that $\pi_{G/H}^{-1}(\psi)$ inherits from $\mathcal{U}_{G/H}$ coincides with the symplectic structure on $T^*(G/[H,H])\sll{\psi}A(H)$.
			\item[\textup{(vi)}] The morphism $\pi_{G/H}$ is surjective, flat, and $G$-invariant.
		\end{itemize}
	\end{proposition}
	
	\begin{proof}
		To prove (i), recall that the fixed points of a group acting by automorphisms on a Poisson algebra form a Poisson subalgebra. This implies that \begin{equation}\mathcal{O}_{\mathcal{U}_{G/H}}=(\sigma_{G/H})_*\left((\mathcal{O}_{T^*(G/[H,H])})^{A(H)}\right)\end{equation} is a sheaf of Poisson algebras. It is also clearly the unique Poisson structure on $\mathcal{U}_{G/H}$ for which $T^*(G/[H,H])\longrightarrow\mathcal{U}_{G/H}$ is a Poisson submersion.
		
		We now verify (ii). Our task is to prove that each symplectic leaf of $\mathcal{U}_{G/H}$ has dimension equal to $2\dim(G/H)$. A first step is to note that the quotient morphism \begin{equation}T^*(G/[H,H])=G\times_{[H,H]}[\h,\h]^{\circ}\longrightarrow G\times_H[\h,\h]^{\circ}=\mathcal{U}_{G/H}\end{equation} is a principal $A(H)$-bundle in the Euclidean topology. A combination of (i) and \cite[Proposition 1.33]{CFMLectures}\footnote{A necessary condition for applying this result is for the $A(H)$-action on $T^*(G/[H,H])$ to be proper. To this end, the quotient map $G/[H,H]\longrightarrow G/H$ is a principal $A(H)$-bundle in the Euclidean topology. One concludes that $A(H)$ acts properly on $G/[H,H]$. It remains only to observe that cotangent lifts of proper Lie group actions are proper.} then reveals that the symplectic leaves of $\mathcal{U}_{G/H}$ are the connected components of the submanifolds $T^*(G/[H,H])\sll{\psi}A(H)$ in the Euclidean topology, where $\psi$ runs over the elements of $\mathfrak{a}(\h)^*$. Since $A(H)$ is abelian and acts freely on $T^*(G/[H,H])$, each of these symplectic leaves has dimension \begin{align}\dim T^*(G/[H,H])-2\dim A(H) & =2\dim G-2\dim [H,H]-2(\dim H-\dim [H,H]) \\ &= 2\dim G-2\dim H\\ &=2\dim(G/H).\end{align}
		
		Parts (iii) and (iv) follow immediately from (i), the fact that the Hamiltonian actions of $G$ and $A(H)$ on $T^*(G/[H,H])$ commute with one another, and the formula in \eqref{Equation: Second double} for the $G$-moment map $\mu_G:T^*(G/[H,H])\longrightarrow\g^*$.
		
		To prove (v), recall that $\pi_{G/H}$ is obtained by descending $\mu_{A(H)}$ to $T^*(G/[H,H])/A(H)=\mathcal{U}_{G/H}$. It follows that \begin{equation}\pi_{G/H}^{-1}(\xi)=\mu_{A(H)}^{-1}(\xi)/A(H)=T^*(G/[H,H])\sll{\xi}A(H)\end{equation} for all $\xi\in\mathfrak{a}(\h)^*$. We may also use \cite[Proposition 1.33]{CFMLectures} to conclude that $T^*(G/[H,H])\sll{\xi}A(H)$ is a Poisson subvariety of $\mathcal{U}_{G/H}$ when equipped with its canonical symplectic structure. These last two sentences imply (v).
		
		Let us now address (vi). To this end, we again recall that $\pi_{G/H}$ is obtained by letting $\mu_{A(H)}$ descend to $T^*(G/[H,H])/A(H)=\mathcal{U}_{G/H}$. The surjectivity assertion now follows immediately from the formula \eqref{Equation: Second double} defining $\mu_{A(H)}$. This formula and \eqref{Equation: Lifted action} also imply the $G$-invariance assertion. As for flatness, we may invoke (v) and the freeness of the $A(H)$-action on $T^*(G/[H,H])$. These facts imply that our surjective morphism $\pi_{G/H}$ has equi-dimensional fibers. We also know that $\pi_{G/H}$ is a morphism between smooth, irreducible varieties. An application of the ``miracle flatness criterion'' \cite[Theorem 23.1]{Matsumura} then forces $\pi_{G/H}$ to be flat.
	\end{proof}
	
	Let us discuss the stronger conclusions that can be obtained if one takes $H$ to be connected. In this case, one finds that \begin{equation}\label{Equation: Connectedness}(\mathfrak{h}^*)^H=(\mathfrak{h}^*)^{\h}=\mathrm{ann}_{\h^*}([\h,\h])=\mathfrak{a}(\h)^*.\end{equation} Lisiecki's bijection \eqref{Equation: Lisiecki} then allows us to write \begin{equation}\mathcal{A}\mathcal{H}\mathcal{L}_G(G/H)=(\h^*)^H=\mathfrak{a}(\h)^*;\end{equation} these identifications are made extensively and implicitly in the following result and proof.
	
	\begin{theorem}\label{Theorem: Universal}
		If $H$ is connected, then $(\mathcal{U}_{G/H},\mu_{G/H},\pi_{G/H},\sigma_{G/H})$ is a universal family of AHL $G$-bundles over $G/H$.
	\end{theorem}
	
	\begin{proof}
		In light of Proposition \ref{Proposition: Prelim}, we only need to verify (iii) and (v) in Definition \ref{Definition: Universal}. Let us therefore fix $\psi\in\mathfrak{a}(\h)^*=\mathcal{A}\mathcal{H}\mathcal{L}_G(G/H)$. Recall the AHL $G$-bundle structure of $(T^*(G/H))^{\psi}$ from Subsection \ref{Subsection: A Lie-theoretic class}. Establishing (iii) and (v) then reduces to showing that $\pi_{G/H}^{-1}(\psi)$ is a symplectic leaf of $\mathcal{U}_{G/H}$, that \begin{equation}\left(\pi_{G/H}^{-1}(\psi),\mu_{G/H}\bigg\vert_{\pi_{G/H}^{-1}(\psi)},\sigma_{G/H}\bigg\vert_{\pi_{G/H}^{-1}(\psi)}\right)\end{equation} is an AHL $G$-bundle over $G/H$, and that this bundle is isomorphic to $(T^*(G/H))^{\psi}$.
		
		We first prove that $\pi_{G/H}^{-1}(\psi)$ is a symplectic leaf of $\mathcal{U}_{G/H}$. Proposition \ref{Proposition: Prelim}(v) tells us that $\pi_{G/H}^{-1}(\psi)=T^*(G/[H,H])\sll{\psi}A(H)$. On the other hand, \cite[Proposition 1.33]{CFMLectures} implies that $T^*(G/[H,H])\sll{\psi}A(H)$ is a symplectic leaf of $\mathcal{U}_{G/H}$ if and only if it is connected in the Euclidean topology. It therefore suffices to prove that $T^*(G/[H,H])\sll{\psi}A(H)$ is connected. This follows immediately from the observation that \begin{equation}T^*(G/[H,H])\sll{\psi}A(H)=G\times_H\mathrm{r}^{-1}(-\psi)\subset G\times_H[\h,\h]^{\circ}=\mathcal{U}_{G/H},\end{equation} where $\mathrm{r}:\g^*\longrightarrow\h^*$ is the restriction map.
		
		To address the balance of the proof, note that
		\begin{equation}\pi_{G/H}^{-1}(\psi)=T^*(G/[H,H])\sll{\psi}A(H)=\left(T^*G\sll{0}[H,H]\right)\sll{\psi}A(H)=T^*G\sll{\psi}H=T^*(G/H)^{\psi}\end{equation} as symplectic $G$-varieties. It is also clear that $\mu_{G/H}\big\vert_{\pi_{G/H}^{-1}(\psi)}$ (resp. $\sigma_{G/H}\big\vert_{\pi_{G/H}^{-1}(\psi)}$) coincides with the moment map \eqref{Equation: Twisted moment} (resp. Lagrangian fibration \eqref{Equation: Twisted fibration}) underlying the AHL $G$-bundle structure of $T^*(G/H)^{\psi}$. We conclude that \begin{equation}\left(\pi_{G/H}^{-1}(\psi),\mu_{G/H}\bigg\vert_{\pi_{G/H}^{-1}(\psi)},\sigma_{G/H}\bigg\vert_{\pi_{G/H}^{-1}(\psi)}\right)\end{equation} is an AHL $G$-bundle over $G/H$, and that this bundle is isomorphic to $(T^*(G/H))^{\psi}$. This completes the proof.
	\end{proof}
	
	\section{Incidence varieties}\label{Section: Incidence varieties}
	We now concern ourselves with incidence-theoretic aspects of Section \ref{Section: Lie-theoretic}. In Subsection \ref{Subsection: Self-normalizing}, we define and give examples of self-normalizing conjugacy classes of closed subgroups. Subsection \ref{Subsection: Variety structure} then establishes that such classes have canonical algebraic variety structures. This allows us to define an incidence-theoretic, regular Poisson variety $\mathcal{U}_{\mathcal{C}}$ for each self-normalizing conjugacy class $\mathcal{C}$.
	
	\subsection{Self-normalizing conjugacy classes}\label{Subsection: Self-normalizing}
	Let $G$ be an affine algebraic group. Note that $G$ acts on the set of its closed subgroups by conjugation. We refer to an orbit of this action as a \textit{conjugacy class} of closed subgroups of $G$. On the other hand, let $N_G(H)\subset G$ denote the normalizer of a closed subgroup $H\subset G$. Recall that $H$ is called \textit{self-normalizing} if $N_G(H)=H$. It is clear that the set of self-normalizing closed subgroups of $G$ is a union of conjugacy classes. A conjugacy class occurs in this union if and only if it contains a self-normalizing element. This discussion can be formalized as follows. 
	
	\begin{definition}
		Let $\mathcal{C}$ be a conjugacy class of closed subgroups of $G$. We call $\mathcal{C}$ \textit{self-normalizing} if $N_G(H)=H$ for all $H\in\mathcal{C}$.
	\end{definition}
	
	The following are some pertinent examples of self-normalizing closed subgroups. By definition, the conjugacy classes of these subgroups are self-normalizing.
	
	\begin{example}[Parabolic subgroups]\label{Example: Parabolic}
		Suppose that $G$ is connected and semisimple. Recall that a closed subgroup $P\subset G$ is called \textit{parabolic} if the variety $G/P$ is projective. The parabolic subgroups of $G$ constitute finitely many conjugacy classes. At the same time, parabolic subgroups are known to be self-normalizing.
	\end{example}
	
	\begin{example}[Normalizers of symmetric subgroups]\label{Example: Normalizers}
		Suppose that $G$ is connected, simply-connected, and semisimple. A closed subgroup $H\subset G$ is called \textit{symmetric} if $H=G^{\sigma}$ for an involutive algebraic group automorphism $\sigma:G\longrightarrow G$. In this case, $N_G(H)$ is self-normalizing and $N_G(H)/H$ is finite; see \cite[Section 2]{mjt:21} and \cite[Lemma 2.9]{Brion}.
	\end{example}
	
	\subsection{The variety structure on a self-normalizing conjugacy class}\label{Subsection: Variety structure} Let $\mathcal{C}$ be a conjugacy class of closed subgroups of $G$. Suppose that $H_1,H_2\in\mathcal{C}$, and let $g\in G$ be such that $H_2=gH_1g^{-1}$. It follows that
	\begin{equation}\label{Equation: Clear}
		G/H_1\longrightarrow G/H_2,\quad [k]\mapsto [kg^{-1}]
	\end{equation}
	is a well-defined $G$-variety isomorphism.
	
	\begin{proposition}
		Suppose that $\mathcal{C}$ is a self-normalizing conjugacy class of closed subgroups of $G$. If $H_1,H_2\in\mathcal{C}$, then the $G$-variety isomorphism \eqref{Equation: Clear} does not depend on the choice of $g\in G$ satisfying $H_2=gH_1g^{-1}$.
	\end{proposition}
	
	\begin{proof}
		Suppose that $g,\ell\in G$ satisfy $gH_1g^{-1}=H_2=\ell H_1\ell^{-1}$. It follows that $\ell=gh$ for some $h\in N_G(H_1)=H_1$. This implies that $gh^{-1}g^{-1}\in H_2$. Given $k\in G$, we conclude that
		\begin{equation}[k\ell^{-1}]=[k(gh)^{-1}]=[kg^{-1}(gh^{-1}g^{-1})]=[kg^{-1}]\in G/H_2.\end{equation} The proof is complete.
	\end{proof}
	
	Let $\mathcal{C}$ be a conjugacy class of closed subgroups of $G$. Given $H\in\mathcal{C}$, the map \begin{equation}\phi_H:G/N_G(H)\longrightarrow\mathcal{C},\quad [g]\mapsto gHg^{-1}\end{equation} is a $G$-equivariant bijection. We may equip $\mathcal{C}$ with the algebraic variety structure for which $\phi_H$ is a variety isomorphism. If $\mathcal{C}$ is self-normalizing, then this structure turns out to be independent of $H$.
	
	\begin{proposition}\label{Proposition: Canonical structure on conjugacy classes}
		Let $\mathcal{C}$ be a self-normalizing conjugacy class of closed subgroups of $G$. There exists a unique algebraic variety structure on $\mathcal{C}$ such that the $G$-equivariant bijection $\phi_H$ is a variety isomorphism for all $H\in\mathcal{C}$. 
	\end{proposition}
	
	\begin{proof}
		Uniqueness follows immediately from the fact that $\phi_H$ is a bijection for all $H\in\mathcal{C}$. To establish existence, suppose that $H_1,H_2\in\mathcal{C}$. It suffices to prove that the identity map $\mathrm{id}_{\mathcal{C}}:\mathcal{C}\longrightarrow\mathcal{C}$ is a variety isomorphism, where the domain (resp. codomain) $\mathcal{C}$ carries the unique variety structure for which $\phi_{H_1}$ (resp. $\phi_{H_2}$) is an isomorphism. Let $\psi:G/H_1\longrightarrow G/H_2$ be the canonical $G$-variety isomorphism \eqref{Equation: Clear}. We have $\mathrm{id}_{\mathcal{C}}=\phi_{H_2}\circ\psi\circ\phi_{H_1}^{-1}$, completing the proof. 
	\end{proof}
	
	\subsection{The regular Poisson variety $\mathcal{U}_{\mathcal{C}}$}\label{Subsection: The variety} Let $\mathcal{C}$ be a self-normalizing conjugacy class of closed subgroups of $G$. Consider the set
	\begin{equation}\label{Equation: Major definition}\mathcal{U}_{\mathcal{C}}\coloneqq\{(H,\xi)\in\mathcal{C}\times\g^*:\xi\in[\h,\h]^{\circ}\},\end{equation} where $\mathfrak{h}\subset\g$ denotes the Lie algebra of $H\in\mathcal{C}$. Note that $\mathcal{U}_{\mathcal{C}}$ is invariant under the diagonal action of $G$ on $\mathcal{C}\times\mathfrak{g}^*$. We also observe that the map
	\begin{equation}\label{Equation: Moment condition}\mu_{\mathcal{C}}:\mathcal{U}_{\mathcal{C}}\longrightarrow\g^*,\quad (H,\xi)\mapsto\xi\end{equation} is $G$-equivariant.
	
	Suppose that $H\in\mathcal{C}$. Recall from Proposition \ref{Proposition: Prelim} that $\mathcal{U}_{G/H}=G\times_H[\h,\h]^{\circ}$ is a Poisson Hamiltonian $G$-variety with moment map $\mu_{G/H}:\mathcal{U}_{G/H}\longrightarrow\g^*$, where $\mu_{G/H}$ is defined in \eqref{Equation: Moment map definition}. Using the fact that $H$ is self-normalizing, a straightforward exercise reveals that
	\begin{equation}\label{Equation: Equivariant bijection}\psi_H:\mathcal{U}_{G/H}\longrightarrow\mathcal{U}_{\mathcal{C}},\quad [g:\xi]\mapsto(gHg^{-1},\mathrm{Ad}_g^*(\xi))\end{equation} is a $G$-equivariant bijection. It is clear that the diagram
	\begin{equation}\label{Equation: Obvious CD}\begin{tikzcd}
			\mathcal{U}_{G/H}\arrow[r, "\psi_H", "\cong"'] \arrow[dr, swap, "\mu_{G/H}"] & \mathcal{U}_{\mathcal{C}} \arrow[d, "\mu_{\mathcal{C}}"] \\
			& \g^*
	\end{tikzcd}\end{equation} commutes.
	
	\begin{theorem}\label{Proposition: New}
		Let $\mathcal{C}$ be a self-normalizing conjugacy class of closed subgroups of $G$.
		\begin{itemize}
			\item[\textup{(i)}] There exists a unique Poisson variety structure on $\mathcal{U}_{\mathcal{C}}$ such that \eqref{Equation: Equivariant bijection} is a Poisson variety isomorphism for all $H\in\mathcal{C}$.
			\item[\textup{(ii)}] The Poisson variety $\mathcal{U}_{\mathcal{C}}$ in \textup{(i)} is regular of rank $2\dim\mathcal{C}$.
			\item[\textup{(iii)}] The $G$-action on $\mathcal{U}_{\mathcal{C}}$ is algebraic and Hamiltonian with respect to the Poisson variety structure in \textup{(i)}.
			\item[\textup{(iv)}] The map $\mu_{\mathcal{C}}$ is a moment map for the Hamiltonian $G$-action in \textup{(iii)}.
		\end{itemize}
	\end{theorem}
	
	\begin{proof}
		Assume that (i) holds. Parts (ii), (iii), and (iv) then follow immediate from Proposition \ref{Proposition: Prelim}(ii), Proposition \ref{Proposition: Prelim}(iii), and the commutativity of \eqref{Equation: Obvious CD}, respectively. It therefore suffices to prove (i). The uniqueness assertion in (i) follows from the fact that \eqref{Equation: Equivariant bijection} is a bijection. To establish existence, suppose that $H_1,H_2\in\mathcal{C}$. It suffices to prove that \begin{equation}\psi_{H_2}^{-1}\circ\psi_{H_1}:\mathcal{U}_{G/H_1}\longrightarrow\mathcal{U}_{G/H_2}\end{equation} is a Poisson variety isomorphism. 
		
		Choose $h\in G$ satisfying $H_2=hH_1h^{-1}$. One finds that 
		\begin{equation}(\psi_{H_2}^{-1}\circ\psi_{H_1})([g:\xi])=[gh^{-1}:\mathrm{Ad}_h^*(\xi)]\end{equation} for all $[g:\xi]\in\mathcal{U}_{G/H_1}$. We have $[H_2,H_2]=h[H_1,H_1]h^{-1}$, implying that \begin{equation}G/[H_1,H_1]\longrightarrow G/[H_2,H_2]\quad [g]\longrightarrow[gh^{-1}]\end{equation} is a variety isomorphism. As such, it induces an isomorphism of cotangent bundles $T^*(G/[H_1,H_1])\overset{\cong}\longrightarrow T^*(G/[H_2,H_2])$. This isomorphism is a symplectomorphism. On the other hand, use \eqref{Equation: Statement} to write $T^*(G/[H_1,H_1])=G\times_{[H_1,H_1]}[\h_1,\h_1]^{\circ}$ and $T^*(G/[H_2,H_2])=G\times_{[H_2,H_2]}[\h_2,\h_2]^{\circ}$. The aforementioned symplectomorphism is given by
		\begin{equation}\label{Equation: New cotangent isomorphism}T^*(G/[H_1,H_1])\overset{\cong}\longrightarrow T^*(G/[H_2,H_2]),\quad [g:\xi]\mapsto [gh^{-1}:\mathrm{Ad}_{h}^*(\xi)].\end{equation} We therefore have a commutative diagram
		\begin{equation}\begin{tikzcd}
				T^*(G/[H_1,H_1])\arrow[r, "\eqref{Equation: New cotangent isomorphism}"] \arrow[d] & T^*(G/[H_2,H_2]) \arrow[d]\\
				\mathcal{U}_{G/H_1} \arrow[r, swap, "\psi_{H_2}^{-1}\circ\psi_{H_1}"] & \mathcal{U}_{G/H_2}
			\end{tikzcd},\end{equation}
		where the left and right vertical maps are the quotients by $A(H_1)$ and $A(H_2)$, respectively. This combines with Proposition \ref{Proposition: Prelim}(i) and the fact that \eqref{Equation: New cotangent isomorphism} is a symplectomorphism to imply that $\psi_{H_2}^{-1}\circ\psi_{H_1}$ is a Poisson variety isomorphism.  
	\end{proof}
	
	\section{Example: Conjugacy classes of normalizers of symmetric subgroups}\label{Section: Symmetric} We now witness examples of Theorem \ref{Proposition: New}. In Subsection \ref{Subsection: New assumptions}, we outline the new Lie-theoretic assumptions made for the balance of this manuscript. This creates context for Subsection \ref{Subsection: Conjugacy classes involutive}, where we associate a self-normalizing conjugacy class of closed subgroups to a conjugacy class of involutive automorphisms of $G$. Concrete applications to regular Poisson varieties are discussed in Subsection \ref{Subsection: Explicit examples}.
	
	\subsection{New assumptions}\label{Subsection: New assumptions} For the duration of this manuscript, $G$ is a connected semisimple affine algebraic group. These hypotheses imply that the Killing form is non-degenerate. We use this fact to freely identify $\g$ and $\g^*$. We also write $V^{\perp}\subset\g$ for the annihilator of a subspace $V\subset\g$ under the Killing form. Note that the identification of $\g$ with $\g^*$ identifies $V^{\perp}\subset\g$ with $V^{\circ}\subset\g^*$.
	
	\subsection{Conjugacy classes of involutive automorphisms}\label{Subsection: Conjugacy classes involutive}
	Let $\mathrm{Aut}(G)$ denote the set of algebraic group automorphisms of $G$. Note that $G$ acts on $\mathrm{Aut}(G)$ by \begin{equation}\label{Equation: Abstract action}(g\cdot\sigma)(h)=g\sigma(g^{-1}hg)g^{-1}\end{equation} for all $g,h\in G$ and $\sigma\in\mathrm{Aut}(G)$. We call an orbit of this action a \textit{conjugacy class of automorphisms}. In other words, two automorphisms lie in the same conjugacy class if and only if they agree up to conjugation by an inner automorphism. The subset of involutions in $\mathrm{Aut}(G)$ is a union of conjugacy classes. We call a conjugacy class $\mathcal{I}$ in this union \textit{involutive}. 
	
	Let $\mathcal{I}\subset\mathrm{Aut}(G)$ be an involutive conjugacy class. Assume that $G$ is simply-connected. By Example \ref{Example: Normalizers}, $N_G(G^{\sigma})\subset G$ is self-normalizing and such that $N_G(G^{\sigma})/G^{\sigma}$ is finite for all $\sigma\in\mathcal{I}$. It follows that 
	\begin{equation}\mathcal{C}_{\mathcal{I}}\coloneqq\{N_G(G^{\sigma}):\sigma\in\mathcal{I}\}\end{equation} is a self-normalizing conjugacy class of closed subgroups of $G$. Set $\mathcal{U}_{\mathcal{I}}\coloneqq\mathcal{U}_{\mathcal{C}_{\mathcal{I}}}$ and note that
	\begin{equation}\mathcal{U}_{\mathcal{I}}=\{(H,x)\in\mathcal{C}_{\mathcal{I}}\times\g:x\in[\h,\h]^{\perp}\},\end{equation} where $\h\subset\g$ is the Lie algebra of $H\in\mathcal{C}_{\mathcal{I}}$. The Poisson Hamiltonian $G$-variety structure on $\mathcal{U}_{\mathcal{I}}$ is explained in Theorem \ref{Proposition: New}. In what follows, we discuss a concrete example.
	
	\subsection{Concrete examples}\label{Subsection: Explicit examples} In this subsection, $G=\operatorname{SL}_n$ for $n\geq 2$. Note that $(\operatorname{SL}_n)^{\sigma}=\operatorname{SO}_n$ for the involutive automorphism \begin{equation}\sigma:\operatorname{SL}_n\longrightarrow\operatorname{SL}_n,\quad g\mapsto (g^T)^{-1}.\end{equation} Let $\mathcal{I}\subset\mathrm{Aut}(\operatorname{SL}_n)$ denote the conjugacy class of $\sigma$. Since $N_{\operatorname{SL}_n}(\operatorname{SO}_n)/\operatorname{SO}_n$ is finite, one has $\h=[\h,\h]\cong\mathfrak{so}_n$ for every $H\in\mathcal{C}_{\mathcal{I}}$ with Lie algebra $\h\subset\g$. This makes it clear that
	\begin{equation}\mathcal{U}_{I}=\{(H,x)\in\mathcal{C}_{\mathcal{I}}\times\mathfrak{sl}_n:x\in\h^{\perp}\}.\end{equation} Theorem \ref{Proposition: New} implies that the map \begin{equation}\operatorname{SL}_n\times_{N_{\mathrm{SL_n}}(\operatorname{SO}_n)}\mathfrak{so}_n^{\perp}\longrightarrow\mathcal{U}_{\mathcal{I}},\quad [g:x]\mapsto (gN_{\mathrm{SL_n}}(\operatorname{SO}_n) g^{-1},gxg^{-1})\end{equation} is an isomorphism of Hamiltonian $\operatorname{SL}_n$-varieties. A similar approach applies if $n$ is even and $\mathcal{I}$ is a conjugacy class of a involutive automorphisms of $\operatorname{SL}_n$ that realizes the symplectic group $\operatorname{Sp}_n\subset\operatorname{SL}_n$ as $(\operatorname{SL}_n)^{\sigma}$.
	
	\section{Example: Conjugacy classes of parabolic subgroups}\label{Section: Twisted}
	We now specialize to the case of a conjugacy class $\mathcal{C}$ of parabolic subgroups of $G$. In Subsection \ref{Subsection: Specifics}, we realize simplifications of the Hamiltonian $G$-variety structure on $\mathcal{U}_{\mathcal{C}}$. The partial Grothendieck--Springer resolution $\mu_{\mathcal{C}}:\g_{\mathcal{C}}\longrightarrow\g$ is formally introduced in Subsection \ref{Subsection: Relation}. In Subsection \ref{Subsection: Embedded}, we prove that that the embedding of $\mathcal{U}_{\mathcal{C}}$ into $\g_{\mathcal{C}}$ is on the level of Hamiltonian $G$-varieties.
	
	\subsection{Specifics on $\mathcal{U}_{\mathcal{C}}$}\label{Subsection: Specifics}
	Let $\mathcal{C}$ be a conjugacy class of parabolic subgroups of $G$. As discussed in Example \ref{Example: Parabolic}, $\mathcal{C}$ is self-normalizing. Suppose that $P\in\mathcal{C}$, and let $\p\subset\g$ denote the Lie algebra of $P$. The Killing form induces a $P$-module isomorphism between $[\p,\p]^{\circ}\subset\g^*$ and $[\p,\p]^{\perp}\subset\g$. In this way, \eqref{Equation: Nicest} and \eqref{Equation: Moment condition} become
	\begin{equation}\label{Equation: Nicest2}\mathcal{U}_{G/P}\coloneqq T^*(G/[P,P])/A(P)=G\times_P[\p,\p]^{\perp}\end{equation} and $$\mu_{G/P}:\mathcal{U}_{G/P}\longrightarrow\g,\quad [g:x]\mapsto\mathrm{Ad}_g(x),$$ respectively. The same principle allows us to write \eqref{Equation: Major definition} as \begin{equation}\mathcal{U}_{\mathcal{C}}\coloneqq\{(P,x)\in\mathcal{C}\times\g:x\in[\p,\p]^{\perp}\}.\end{equation} The Hamiltonian $G$-action and moment map on $\mathcal{U}_{\mathcal{C}}$ are then given by
	\begin{equation}g\cdot(P,x)\coloneqq (gPg^{-1},\mathrm{Ad}_g(x)),\quad g\in G,\text{ }(P,x)\in\mathcal{U}_{\mathcal{C}}\end{equation} and \begin{equation}\mu_{\mathcal{C}}:\mathcal{U}_{\mathcal{C}}\longrightarrow\g,\quad (P,x)\mapsto x,\end{equation} respectively. Given $P\in\mathcal{C}$, \eqref{Equation: Equivariant bijection} takes the form 
	\begin{equation}\label{Equation: Exceptional isomorphism}\psi_P:\mathcal{U}_{G/P}\longrightarrow\mathcal{U}_{\mathcal{C}},\quad [g:x]\mapsto (gPg^{-1},\mathrm{Ad}_g(x)).\end{equation}
	
	\subsection{Partial Grothendieck--Springer resolutions}\label{Subsection: Relation}
	Let $\mathcal{C}$ be a conjugacy class of parabolic subgroups of $G$. Consider the closed, $G$-invariant subvariety of $\mathcal{C}\times\g$ defined by
	\begin{equation}\g_{\mathcal{C}}\coloneqq\{(P,x)\in\mathcal{C}\times\g:x\in\p\},\end{equation} where $\p\subset\g$ is the Lie algebra of $P\in\mathcal{C}$. Let us also consider the $G$-equivariant morphism \begin{equation}\varphi_{\mathcal{C}}:\g_{\mathcal{C}}\longrightarrow\g,\quad(P,x)\mapsto x.\end{equation}
	
	\begin{definition}
		The morphism $\varphi_{\mathcal{C}}:\g_{\mathcal{C}}\longrightarrow\g$ is called the \textit{partial Grothendieck--Springer resolution} associated to $\mathcal{C}$.
	\end{definition}
	
	Grothendieck--Springer resolutions are prominent at the interface of Lie theory, Poisson geometry, and geometric representation theory \cite{CM,Leslie,Safronov,Schrader,Chriss}.
	
	The approach used to endow $\mathcal{U}_{\mathcal{C}}$ with a Poisson Hamiltonian $G$-variety structure applies analogously to $\g_{\mathcal{C}}$. To this end, let $P\subset G$ be a parabolic subgroup integrating $\p\in\mathcal{C}$. Write $U(P)\subset P$ for the unipotent radical of $P$, let $\mathfrak{u}(\p)\subset\g$ denote its Lie algebra, and set $L(P)\coloneqq P/U(P)$. It follows that $T^*(G/U(P))$ is a Hamiltonian $(G\times L(P))$-variety. The geometric quotient $T^*(G/U(P))/L(P)$ exists, is smooth, and inherits a Poisson Hamiltonian $G$-variety structure from the symplectic Hamiltonian $G$-variety structure of $T^*(G/U(P))$ (cf. Proposition \ref{Proposition: Prelim}). At the same time, $\p=\mathfrak{u}(\p)^{\perp}$. This fact yields a commutative diagram
	\begin{equation}\label{Equation: CD6}\begin{tikzcd}
			T^*(G/U(P))\ar[-,double line with arrow={-,-}]{r} \arrow[d] & G\times_{U(P)}\p \arrow[d] & \\
			T^*(G/U(P))/L(P)\ar[-,double line with arrow={-,-}]{r} & G\times_P\p \arrow[r, "\cong", "\eqref{Equation: Equivariant bundle2}"'] & \g_{\mathcal{C}}
		\end{tikzcd},\end{equation} in which the horizontal arrow is the $G$-variety isomorphism
	\begin{equation}\label{Equation: Equivariant bundle2}
		T^*(G/U(P))/L(P)=G\times_P\p\overset{\cong}\longrightarrow\g_{\mathcal{C}},\quad [g:x]\mapsto (gPg^{-1},\mathrm{Ad}_g(x)). 
	\end{equation}  
	There is a unique Poisson Hamiltonian $G$-variety structure on $\g_{\mathcal{C}}$ such that \eqref{Equation: Equivariant bundle2} is an isomorphism for all $P\in\mathcal{C}$. The Grothendieck--Springer resolution $\varphi_{\mathcal{C}}:\g_{\mathcal{C}}\longrightarrow\g$ is a moment map for this structure. 
	
	\subsection{A relationship between $\mathcal{U}_{\mathcal{C}}$ and $\g_{\mathcal{C}}$}\label{Subsection: Embedded} Observe that $[\p,\p]^{\perp}\subset\p$ for all parabolic subalgebras $\p\subset\g$. This fact implies that $\mathcal{U}_{\mathcal{C}}$ is a closed subvariety of $\g_{\mathcal{C}}$, where $\mathcal{C}$ is a conjugacy class of parabolic subgroups of $\g$.
	
	\begin{theorem}\label{Proposition: Embedding}
		The inclusion morphism $\mathcal{U}_{\mathcal{C}}\longrightarrow\g_{\mathcal{C}}$ is a $G$-equivariant embedding of Poisson varieties. It also makes the diagram \begin{equation}\label{Equation: CD4}\begin{tikzcd}
				\mathcal{U}_{\mathcal{C}}\arrow[r] \arrow[dr, swap, "\mu_{\mathcal{C}}"] & \g_{\mathcal{C}} \arrow[d, "\varphi_{\mathcal{C}}"] \\
				& \g
		\end{tikzcd}\end{equation}
		commute.
	\end{theorem}
	
	\begin{proof}
		The equivariance and commutativity assertions follow from straightforward calculations. It therefore suffices to prove that $\mathcal{U}_{\mathcal{C}}$ is a union of symplectic leaves of $\g_{\mathcal{C}}$, and that each such leaf is a symplectic leaf of $\mathcal{U}_{\mathcal{C}}$. To this end, choose $\p\in\mathcal{C}$ and let $P\subset G$ be the parabolic subgroup integrating $\p$. We have the commutative diagram \begin{equation}\label{Equation: CD7}\begin{tikzcd}
				\mathcal{U}_{G/P}\ar[-,double line with arrow={-,-}]{r} & G\times_{P}[\p,\p]^{\perp}\arrow[r, "\eqref{Equation: Exceptional isomorphism}", "\cong"'] \arrow[d] & \mathcal{U}_{\mathcal{C}} \arrow[d] \\
				T^*(G/U(P))/L(P)\ar[-,double line with arrow={-,-}]{r} & G\times_P\p \arrow[r, "\cong", "\eqref{Equation: Equivariant bundle2}"'] & \g_{\mathcal{C}}
			\end{tikzcd},\end{equation} where the vertical maps are inclusions. Let us also recall that the horizontal arrows are Poisson variety isomorphisms. It therefore suffices to prove the following: $\mathcal{U}_{G/P}=G\times_P[\p,\p]^{\perp}$ is a union of symplectic leaves of $T^*(G/U(P))/L(P)=G\times_P\p$, and each such leaf is a symplectic leaf with respect to the Poisson structure on $T^*(G/U(P))/L(P)=G\times_P\p$. 
		
		The proof of Theorem \ref{Theorem: Universal} tells us that the symplectic leaves of $\mathcal{U}_{G/P}$ are the symplectic quotients $T^*(G/[P,P])\sll{\psi}A(P)$, where $\psi$ runs over the elements of $\mathfrak{a}(\p)^*=(\p^*)^P$. On the other hand, let $\mathfrak{l}(\p)\coloneqq\p/\mathfrak{u}(\p)$ denote the Lie algebra of $L(P)$. A similar approach reveals that the symplectic leaves of $T^*(G/U(P))/L(P)$ are the symplectic quotients $T^*(G/U(P))\sll{\psi}L(P)$, where $\psi$ runs over the elements of $\mathfrak{l}(\p)^*=\mathrm{ann}_{\p^*}(\mathfrak{u}(\p))$. We also have the quotient group $I(P)\coloneqq [P,P]/U(P)$ and exact sequences
		\begin{equation}\{e\}\longrightarrow U(P)\longrightarrow[P,P]\longrightarrow I(P)\longrightarrow\{e\}\end{equation} and \begin{equation}\{e\}\longrightarrow I(P)\longrightarrow L(P)\longrightarrow \bunderbrace{L(P)/I(P)}{\let\scriptstyle\textstyle\cong A(P)}\longrightarrow \{e\}.\end{equation}
		Given any $\psi\in(\p^*)^P$, it follows that
		\begin{align}T^*(G/[P,P])\sll{\psi}A(P) & = (T^*G\sll{0}[P,P])\sll{\psi}A(P)\\
			& = \bigg((T^*G\sll{0}U(P))\sll{0}I(P)\bigg)\sll{\psi}A(P)\\
			& = T^*(G/U(P))\sll{\psi}L(P).
		\end{align}
		This verifies the last sentence of the previous paragraph, completing the proof.
	\end{proof}
	
	\bibliographystyle{acm} 
	\bibliography{Incidence}
	
\end{document}